\documentclass[12pt]{amsart}
\usepackage{amsmath, amsthm, amscd, amsfonts, amssymb}
\newtheorem{theorem}{Theorem}[section]

\newtheorem{Lemma}[theorem]{Lemma}
\theoremstyle{definition}

\newtheorem{remark}[theorem]{Remark}

\newcommand{\psu}{{\rm{PSU}}_3(q)}

\begin{document}

\title[On recognizability of ${\rm{PSU}}_3(q)$ by the orders of maximal abelian subgroups]
{On recognizability of ${\rm{PSU}}_3(q)$ by the orders\\ of maximal
abelian subgroups}

\author[Zahra Momen\,\, \&\ Behrooz Khosravi\,     ]{Zahra Momen  \ \& \  Behrooz Khosravi  }

\address{ School of Mathematics,
Institute for Research in Fundamental sciences  (IPM), P.O.Box:
19395--5746, Tehran, Iran \newline Dept. of Pure  Math.,  Faculty  of Math. and Computer Sci. \\
Amirkabir University of Technology (Tehran Polytechnic)\\ 424,
Hafez Ave., Tehran 15914, IRAN }

\email{ zahramomen@yahoo.com}

\email{ khosravibbb@yahoo.com}

\thanks{ The second author was supported in part by a grant from IPM  (No 92050120)}



\begin{abstract}
In [Li and Chen, A new characterization of the simple group
$A_1(p^n)$, {\it Sib. Math. J.,} 2012], it is proved that the simple
group $A_1(p^n)$ is uniquely determined by the set of orders of its
maximal abelian subgroups. Also in [Momen and Khosravi, Groups with
the same orders of maximal abelian subgroups as $A_2(q)$, {\it
Monatsh. Math.,} 2013], the authors proved that if $L=A_2(q)$, where
$q$ is not a Mersenne prime, then every finite group with the same
orders of maximal abelian subgroups as $L$, is isomorphic to $L$ or
an extension of $L$ by a subgroup of the outer automorphism group of
$L$. In this paper, we prove that if $L=\psu$, where $q$ is not a
Fermat prime, then every finite group with the same orders of
maximal abelian subgroups as $L$, is isomorphic to $L$ or an
extension of $L$ by a field automorphism of $L$.
\end{abstract}
\maketitle
\section{\bf Introduction}
If $n$ is an integer, then we denote by $\pi(n)$ the set of all
prime divisors of $n$. If $G$ is a finite group, then $\pi(|G|)$
is denoted by $\pi(G)$. We construct the {\it prime graph} of $G$,
which is denoted by $\Gamma(G)$, as follows: the vertex set is
$\pi(G)$ and two distinct primes $p$ and $p'$ are joined by an
edge if and only if $G$ has an element of order $pp'$. Let $s(G)$
be the number of connected components of $\Gamma(G)$ and let
$\pi_{1}(G), \pi_{2}(G), ..., \pi_{s(G)}(G)$ be the connected
components of $\Gamma(G)$. Sometimes we use the notation $\pi_i$
instead of $\pi_i(G)$. If $2\in \pi(G)$ we always suppose that  $2\in
\pi_{1}(G)$. Also $|G|$ can be expressed as a product of coprime
positive integers $m_i(G)$ (or briefly $m_i$), $i=1,2,...,s(G)$ with $\pi(m_i)=\pi_i$.
The set of integers $m_i$, $i=1,2,...,s(G)$ is called the set of {\it order components} of $G$ and is denoted by $OC(G)$. The order components of finite simple groups with disconnected prime
graphs are listed in \cite[Tables 1-3]{ijacb}. Let $m$ and
$n$ be natural numbers.  We write $m\thicksim n$ if and only if
for every prime divisor $r\in \pi(m)$ and $s\in\pi(n)$, $r$ is
adjacent to $s$ in $\Gamma(G)$. The {\it spectrum} of a finite
group $G$,  which is denoted by $\pi_e(G)$, is the set of its
element orders. A subset $X$ of the vertices of a graph is called
an {\it independent set} if the induced subgraph on $X$ has no edge.
Let $G$ be a finite group and $r\in\pi(G)$. We denote by
$\rho(G)$ some independent set of vertices in $\Gamma(G)$ with
the maximal number of elements. Also some independent set of
vertices in $\Gamma(G)$ containing $r$ with the maximal number of
elements is denoted by $\rho(r,G)$. Let $t(G)=|\rho(G)|$ and
$t(r,G)=|\rho(r,G)|$. Let $M(G)=\{|H| :$ $H$ is a maximal abelian subgroup of $G$$\}$.

A simple group $G$ is called a $K_3$-group if $|\pi(G)|$ = 3. It is
known that if $G$ is a $K_3$-group, alternating group $A_n$, where
$n$ and $n-2$ are primes or $n\leq 10$, $M$-group, $J$-group,
$A_{1}(2^n)$, $S_{z}(2^{2m+1})$, a sporadic simple group or the
automorphism group of a sporadic group, then $G$ is uniquely
determined by the orders of its maximal abelian subgroups
{\rm{(}}see \cite{chenan,han,kuma,wangtez,wang}{\rm{)}}. Recently in
\cite{a1pn}, it is proved that if $q$ is a prime power, then the
simple group $A_1(q)$ is uniquely determined by the orders of its
maximal abelian subgroups. In \cite{a2q}, it is proved that if $G$
is a finite group such that $M(G)=M(A_2(q))$, where $q$ is not a
Mersenne prime, then $G$ is isomorphic to $A_2(q)$ or an extension
of $A_2(q)$ by a subgroup of the outer automorphism group of
$A_2(q)$.

In this paper, we consider the simple group ${\rm{PSU}}_3(q)$, where
$q$ is not a Fermat prime, and as the main result we prove that if
$G$ is a finite group such that $M(G)=M({\rm{PSU}}_3(q))$, then $G$
is isomorphic to ${\rm{PSU}}_3(q)$ or an extension of
${\rm{PSU}}_3(q)$ by a subgroup of the outer automorphism group of
${\rm{PSU}}_3(q)$. Note that up to now the only known Fermat primes
are 3, 5, 17, 257, 65537.

Throughout this paper we suppose that $q$ is a prime power. Also all
groups are finite and by simple groups we mean nonabelian simple
groups. If $m$ and $n$ are natural numbers, then $(m,n)$ means the
greatest common divisor of $m$ and $n$ and $[m,n]$ means the least
common multiple of $m$ and $n$. Also if $m$ is a positive integer
and $p$ is a prime number, then $(m)_p$ denotes the $p$-part of $m$,
in other words, $(m)_p= p^k$ if $p^k\mid m$ but $p^{k+1}\nmid m$.
Let $G$ be a group and $H$ be a subgroup of $G$. Then $H$ char $G$
means that $H$ is a characteristic subgroup of $G$. All further
unexplained notations are standard and refer to \cite{atlas}. For
the proof of the main theorem, we use the classification of finite
simple groups.
\section{\bf Preliminary Results}
Using \cite[Lemma 2.2]{ijacb}, we have the following result:
\begin{Lemma}\label{ijac}
A finite group $G$ with disconnected prime graph $\Gamma(G)$ satisfies one of the following conditions:\\
(a) $s(G)=2$ and $G=KC$ is a Frobenius group with kernel $K$ and complement $C$ and two connected components of $\Gamma(G)$ are $\Gamma(K)$ and $\Gamma(C)$. Moreover $K$ is nilpotent, and hence $\Gamma(K)$ is a complete graph. If $C$ is solvable, then $\Gamma(C)$ is complete; otherwise, $\{2,3,5\}\subseteq \pi(G)$ and $\Gamma(C)$ can be obtained from the complete graph with vertex set $\pi(C)$ by removing the edge $\{3,5\}$.\\
(b) $s(G)=2$ and $G$ is a 2-Frobenius group, i.e., $G=ABC$, where $A$ and $AB$ are normal subgroups of $G$, $B$ is a normal subgroup of $BC$, and $AB$ and $BC$ are Frobenius groups. The two connected components of $\Gamma(G)$ are complete graphs $\Gamma(AC)$ and $\Gamma(B)$.\\
(c) $G$ has a normal series $1\unlhd H\unlhd K\unlhd G$ such that $H$ and $G/K$ are $\pi_1$-groups, while $H$ is nilpotent, $K/H$ is a nonabelian simple group, and $|G/K|\mid |{\rm{Aut}}(K/H)|$.
\end{Lemma}
\begin{Lemma}\label{froprime}
{\rm (\cite{frochen}, \cite{pasman})} Let $G$ be a Frobenius group of even order with kernel $K$ and complement $H$. Then $s(G)=2$, the prime graph components of $G$ are $\pi(H)$ and $\pi(K)$ and the following assertions hold:

(a) $K$ is nilpotent;

(b) $|K|\equiv 1 \pmod{|H|}$;

(c) Every subgroup of $H$ of order $ts$, with $t$ and $s$ (not necessarily distinct) primes, is cyclic.

In particular, $2\in\pi(K)$ and all Sylow subgroups of $H$ are cyclic or, $2\in\pi(H)$, $K$ is an abelian group, $H$ is a solvable group, the Sylow subgroups of odd order of $H$ are cyclic and the $2$-Sylow subgroups of $H$ are cyclic or generalized quaternion groups. If $H$ is nonsolvable, then $H$ has a subgroup $H_0$ such that $[H: H_0]\leq 2$, $H_0=Z\times {\rm{SL}}_2(5)$, $(|Z|,30)=1$ and the Sylow subgroups of $Z$ are cyclic.
\end{Lemma}
\begin{Lemma}\label{=ordercom}
{\rm (\cite{tabl14})} Let $G$ be a finite group and let $N$ be a nonabelian simple group. If $s(G)\geq 2$, $M(G)=M(N)$ and $G$ has a normal series $1\unlhd H\unlhd K\unlhd G$ such that $H$ and $G/K$ are $\pi_1$-groups, while $K/H$ is a nonabelian simple group, then

(a) the odd order components of $N$ are equal to some of those of $K/H$. Moreover $s(K/H)\geq s(G)$;

(b) if $H=1$, then $G/K\leq {\rm{Out}}(K)$.
\end{Lemma}
\begin{Lemma}\label{admartabe-1}
{\rm (\cite{tabl14})} Let $G$ be a finite group, $s(G)\geq 2$ and $N\unlhd G$. If $N$ is a $\pi_1$-group, $a_1,a_2,...,a_r$ are the odd order components of $G$, then $a_1 a_2\cdots a_r$ is a divisor of $|N|-1$.
\end{Lemma}
\begin{Lemma}\label{1417malle}
{\rm (\cite[Corollary 14.17]{malleb})} Let $G$ be connected
reductive with simply connected derived subgroup $[G,G]$, and
$s_1,\cdots s_r\in G$ mutually commuting semisimple elements of
finite order. Let $n$ be the number of $s_i$ whose order is not
prime to all torsion primes of $G$. If $n\leq 2$ then there exists a
maximal torus of $G$ containing all $s_i$.
\end{Lemma}
\begin{Lemma}\label{cres1}
{\rm (\cite{cres})} The equation $p^m-q^n=1$, where $p$ and $q$ are primes
and $m$, $n>1$, has only one solution, namely $3^2-2^3=1$.
\end{Lemma}
\begin{Lemma}\label{cres2}
{\rm (\cite{cres})} With the exceptions of the relations $(239)^2-2(13)^4=-1$ and
$3^5-2(11)^2=1$ every solution of the equation $$p^m-2q^n=\pm 1;\ p,\ q\  prime;\  m,n>1$$ has exponents $m=n=2$; i.e. it comes from a unit $p-q\cdot 2^{1/2}$ of the quadratic field
$Q(2^{1/2}$) for which the coefficients $p$ and $q$ are primes.
\end{Lemma}
\begin{Lemma}\label{zsig}
{\rm(Zsigmondy's Theorem)} {\rm(\cite{zsig})} Let $p$ be a prime
and let $n$ be a positive integer. Then one of the following
holds:

(i) There is a primitive prime $p'$ for $p^n-1$, that is, $p'\mid
(p^n-1)$ but $p'\nmid (p^m-1)$, for every $1\leq m<n$,

(ii) $p=2$, $n=1$ or 6,

(iii) $p$ is a Mersenne prime and $n=2$.
\end{Lemma}
\begin{Lemma}\label{bmm}
{\rm (\cite{ire})} Let $a>1$, $m$ and $n$ be positive integers. Then $(a^{n}-1,a^{m}-1)=a^{(m,n)}-1$.
\end{Lemma}
\begin{Lemma}\label{xyu}
{\rm (\cite{ire})} Let $x,y,z$ be natural numbers such that $x\mid
z$, $y\mid z$ and $d=(x,y)$. Then $xy/d\mid z$.
\end{Lemma}
\begin{remark}\label{12}
(\cite{ire}) If $q$ is a natural number, $r$ is an odd prime and
$(q, r) = 1$, then by $e(r, q)$ we denote the smallest natural
number $m$ such that $q^m\equiv 1 \pmod{r}$. Given an odd $q$,
put $e(2, q) = 1$ if $q\equiv1 \pmod{4}$ and put $e(2, q) = 2$ if
$q\equiv -1 \pmod{4}$. Using Fermat's little theorem we can see
that if $r$ is an odd prime such that $r\mid (q^n-1)$, then $e(r,
q)\mid n$.
\end{remark}
\begin{Lemma}\label{torus}
{\rm(\cite[Lemma 1.2]{vasil})} Let $q$ be a prime power and
$d=(3,q+1)$. Then the orders of the maximal tori of the simple group
$\psu$, are as follows: $(q^2-1)/d$, $(q+1)^2/d$ and $(q^2-q+1)/d$.
\end{Lemma}
\section{\bf The Main Results}
\begin{Lemma}\label{omegakh}
Let $L$ and $G$ be finite groups and let $G$ have a normal series
$1\unlhd H\unlhd K\unlhd G$. If $M(G)=M(L)$, then for every element
$m\in\pi_e(K/H)$, there exists an element $n\in M(L)$ such that
$m\mid n$. In particular, $m\mid |L|$.
\end{Lemma}
\begin{proof}
By assumption $K/H$ has an element of order $m$, so $m\in \pi_e(G)$.
Therefore $G$ contains a cyclic subgroup of order $m$, which is
abelian. Hence $m$ is a divisor of some $n\in M(G)$. Since
$M(G)=M(L)$, we get the result and clearly, $m\mid |L|$.
\end{proof}
\begin{remark}\label{2a2graph}
Let $p$ be a prime number. Let $G=\psu$, where $q=p^{\alpha}>2$ and
$q$ is not a Fermat prime. Let $r_i$ be a primitive prime divisor of
$q^i-1$. We know that $|\psu|=q^3 (q^2-1) (q^3+1)/(3,q+1)$. Also
$\pi_1(G)=\pi(q(q^2-1))$ and $\pi_2(G)=\pi((q^3+1)/(q+1)(3,q+1))$.
 If $(q+1)_3=3$,
then $\rho(G)=\rho(p,G)=\{p,3,r_1\neq2,r_6\}$. If $(q+1)_3\neq 3$,
then $\rho(G)=\rho(p,G)=\{p,r_1\neq2,r_6\}$. Also if $q$ is odd,
then $\rho(2,G)=\{2,r_6\}$. {\rm (} see {\cite{atlas,vasil,cocliq}}
{\rm )}
\end{remark}
\begin{remark}\label{psl2,5}
Let $B$ be an abelian subgroup of ${\rm {PSU}}_{n}(q)$ such that
$(|B|,q)=1$. Then it is not always true that $|B|$ divides the order
of some maximal torus of ${\rm {PSU}}_{n}(q)$. For example, in the
simple group ${\rm {PSU}}_{2}(5)$, there is an abelian subgroup of
order $2^2$. On the other hand by \cite{atlas}, the orders of
maximal tori in ${\rm {PSU}}_{2}(5)$ are $2$ and $3$.
\end{remark}
\begin{theorem}\label{malle}
Let $G={\rm{PSU}}_n(q)$ and $d=(n,q+1)$. If $H$ is an abelian
subgroup of $G$ such that $(|H|,qd)=1$, then $|H|$ divides the order
of some maximal torus of $G$.
\end{theorem}
\begin{proof}
By assumption, $H\cong L/N$, for some subgroup $L$ of ${\rm
SU}_n(q)$, where $N=Z({\rm SU}_n(q))$. Since $N$ and $L/N$ are
abelian, $L$ is solvable. In addition $(|L/N|,|N|)=1$. Therefore
there exists a Hall subgroup $H^{*}\leq L$ of order $|L/N|$. Thus
$L=H^*N$ and hence $H^*\cong L/N\cong H$, which implies that there
exists an abelian subgroup of ${\rm SU}_n(q)$ of order $|H|$.

So $H^*$ is a preimage of $H$ inside ${\rm SU}_n(q)$ which is
abelian and of order coprime to $q$, i.e. all of its elements are
semisimple. Now ${\rm SU}_n(q)$ is a subgroup of ${\rm SU}_n(K)$,
where $K$ is the algebraic closure of ${\rm{GF}}(q^2)$. So by Lemma
\ref{1417malle}, $H$ lies in a maximal torus $T$ of ${\rm SU}_n(K)$.
Note that $T$ is a maximal torus of ${\rm SU}_n(K)$ not of the
finite group.

Now we claim that $T$ is $F$-stable, where $F$ is the Frobenius
endomorphism associated to the finite group of Lie type.

The key point is to observe that Lemma \ref{1417malle} follows from
\cite[Theorem 14.16]{malleb} by intersecting centralizers. Now if
$g$ is in ${\rm SU}_n(q)$, then $g$ is $F$-stable, where $F$ is the
relevant Frobenius endomorphism. Clearly the centralizer of $g$ will
also be $F$-stable. Now intersecting two $F$-stable subgroup results
in an $F$-stable subgroup and so the torus one obtains in Lemma
\ref{1417malle}, is also $F$-stable.

Therefore $H$ lies in $T\cap G^F$, where $G^F={\rm SU}_n(q)$ and so
we can conclude that $H$ lies in a maximal torus of ${\rm SU}_n(q)$.

The above argument shows that whenever $(|H|,d)=1$, then the claim
is true, because there will be a preimage of the same order in ${\rm
SU}_n(q)$, and hence is abelian.
\end{proof}
\begin{theorem}\label{d=1}
Suppose that $q=p^{\alpha}>2$, $q\neq 9$ and $q$ is not a Fermat
prime. If $G$ is a finite group such that $M(G)=M(\psu)$, then $G$
has a unique nonabelian composition factor which is isomorphic to
$\psu$.
\end{theorem}
\begin{proof}
In the sequel, we denote by $r_i$ a primitive prime divisor of
$q^i-1$. By assumption, obviously we have $\Gamma(G)=\Gamma(\psu)$.
Also we consider $d=(3,q+1)$.

First we prove that $G$ is neither a Frobenius group nor a
2-Frobenius group. Suppose that $G=FC$ is a Frobenius group with
kernel $F$ and complement $C$. Since $F$ is nilpotent, so
$\Gamma(F)$ is complete. Now if $C$ is solvable, then by Lemma
\ref{froprime}, $\Gamma(C)$ is also complete, which is a
contradiction. So $C$ is nonsolvable, $\pi(F)=\pi_2(G)$ and
$\pi(C)=\pi_1(G)$. Using Lemma \ref{froprime}, there exists $H_0\leq
C$ such that $[C: H_0]\leq 2$, $H_0=Z\times {\rm SL}_2(5)$ and
$(|Z|,30)=1$. On the other hand by Lemma \ref{ijac}, all prime
numbers in $\pi(C)$ are adjacent except $\{3,5\}$. Since $p\nsim
r_1$ in $\pi_1(G)$, hence $p=3$ and the only possibility for $r_1$
is $5$ or $p=5$ and the only possibility for $r_1$ is $3$. Let $p=3$
and the only possibility for $r_1$ is $5$. Since $2\mid
(3^{\alpha}-1)$, so $2\in r_1$, which is a contradiction. If $p=5$
and the only possibility for $r_1$ is $3$, then similarly we get a
contradiction.

If $G$ is a 2-Frobenius group, then by Lemma \ref{ijac}, $\Gamma(G)$
has two complete connected components, which is a contradiction,
since $p\nsim r_1$ in $\Gamma(G)$ by Remark \ref{2a2graph}.

Therefore $G$ is neither a Frobenius group nor a 2-Frobenius group.
So by Lemma \ref{ijac}, $G$ has a normal series $1\unlhd H\unlhd
K\unlhd G$ such that $H$ and $G/K$ are $\pi_1$-groups, while $H$ is
nilpotent and $K/H$ is a nonabelian simple group.

By Lemma \ref{=ordercom}, we have $s(K/H)\geq s(\psu)=2$ and the odd
order component of $\psu$ is an odd order component of $K/H$. Now
using the classification of finite simple groups with disconnected
prime graph
in \cite[Tables 1--3]{ijacb}, we consider each possibility for $K/H$ and we prove that $K/H\cong \psu$.\\
{\bf Case~1.} Let $K/H\cong A_n$, where $n=p', p'+1,p'+2$ and $p'$
is an odd prime. We know that the odd order component(s) of $A_n$
are $p'$ or $p'+2$. By Remark \ref{2a2graph}, the odd order
component of $G$ is $(q^2-q+1)/d$.

First suppose that $d=1$. By Lemma \ref{=ordercom}, $q^2-q+1=p'$ or
$q^2-q+1=p'+2$.

If $q^2-q+1=p'$, then
$\pi(p'-2)=\pi(q^2-q-1)\subseteq\pi_1(G)=\pi(q(q^2-1))$. On the
other hand, $(q^2-q-1,q(q^2-1))=1$. Thus $p'=3$ and so $q=2$, which
is impossible.

If $q^2-q+1=p'+2$, then $\pi(p')=\pi(q^2-q-1)\subseteq\pi_1(G)$.
Similarly to the above $p'=1$, which is a contradiction.

Now Let $d=3$. Lemma \ref{=ordercom} implies that
$(q^2-q+1)/3=m\in\{p',p'+2\}$. We know that
$\pi(m-3)=\pi((q^2-q-8)/3)\subseteq\pi_1(G)=\pi(q(q^2-1))$. Since
$\pi(((q^2-q-8)/3,q(q^2-1)))\subseteq\{2,3\}$, so
$\pi(q^2-q-8)\subseteq\{2,3\}$. We can easily see that $9\nmid
(q^2-q-8)$. Hence $q^2-q-8=3\cdot 2^t$. If  $t=1$, then $q\notin
\Bbb{N}$. Thus $t>1$. Then
$\pi(m-4)=\pi((q^2-q-11)/3)\subseteq\pi(q(q^2-1))$. Similarly to the
above, $\pi(q^2-q-11)\subseteq\{3,11\}$. Therefore
$\pi((q^2-q-11)/3)=\pi((q^2-q-8-3)/3)=\pi((3\cdot
2^t-3)/3)=\pi(2^t-1)\subseteq\{3,11\}$. Since $3\in\pi(2^2-1)$ and
$11\in\pi(2^{10}-1)$, so by Lemma \ref{zsig}, for $t>10$, there
exists a prime divisor $s\neq3,11$ of $2^t-1$, which is a
contradiction. Also since $\pi(2^t-1)\subseteq\{3,11\}$, so $t=2$
and thus $q=5$, which is a contradiction, because we have $q-1\neq
2^k$.\\
{\bf Case~2.} Let $K/H\cong {\rm{PSU}}_{p'}(q')$, where $q'=p_0^{\beta}$ and $p'$ be an odd prime.\\
Thus by Lemma \ref{=ordercom} and Remark \ref{2a2graph},
$(q'^{p'}+1)/(q'+1)d'=(q^2-q+1)/d$, where $d'=(p',q'+1)$. So
$(q'^{p'}+1)/(q'+1)=d'(q^2-q+1)/d$. Subtracting 1 from both sides,
we have
\begin{equation}\label{abboz}
\frac{q'^{p'-1}-1}{q'+1}=\frac{d'q(q-1)+d'-d}{dq'}
\end{equation}
$\bullet$ Let $p'=3$. Therefore $(q'^2-q'+1)/d'=(q^2-q+1)/d$.

If $d=d'$, then $q'(q'-1)=q(q-1)$. Hence $q=q'$ and $K/H\cong \psu$.

If $d\neq d'$, then without less of generality let $d=3$ and $d'=1$.
So $q'^2-q'+1=(q^2-q+1)/3$. Thus $3q'(q'-1)=(q+1)(q-2)$. Since
$d=3$, so $(q+1,q-2)=3$. Hence $q'\mid(q+1)$ or $q'\mid(q-2)$.

If $q'\mid(q+1)$, then $q+1=q'a$, for some $a\in\Bbb{N}$. Therefore
by the last equation, we have $3q'(q'-1)=q'a(q'a-3)$ and so
$q'=3(a-1)/(a^2-3)$. Clearly $a>1$ and since $(a-1)/(a^2-3)\leq 2$,
so $q'\leq 5$. On the other hand $(3,q'+1)=1$, thus $q'\in\{3,4\}$.
If $q'=3$, then $q=5$, which is a contradiction.

If $q'=4$, then $q(q-1)=38$, which is impossible.

If $q'\mid(q-2)$, then similarly to the above, we get a
contradiction.\\
$\bullet$ Let $p'\geq 5$. First suppose that $d=1$.\\
{\bf (a)} Let $d'=(p',q'+1)=1$. Then by Lemma \ref{=ordercom},
\begin{equation}\label{zb1}
q'(q'^{p'-1}-1)=q(q-1)(q'+1).
\end{equation}
If $(q,q')\neq 1$, then by (\ref{zb1}), $q=q'$ and $p'=3$, which is
a contradiction.

If $(q,q')=1$, then put $k=(p'-1)/2$. Therefore by (\ref{zb1}),
$q'(q'^k-1)(q'^k+1)=q(q-1)(q'+1)$. Since $(q'^k-1,q'^k+1)\mid 2$, so
$q\mid(q'^k-1)$ or $q\mid(q'^k+1)$.

Let $q\mid(q'^k-1)$. Thus $q'^k-1=qt_0$, for some $t_0\in\Bbb{N}$.
Hence by (\ref{zb1}), $q't_0(qt_0+2)=(q-1)(q'+1)$. Consequently
$q(q't_0^2-1-q')=-q'(2t_0+1)-1$. If $t_0>1$, then the right-hand
side of the equality is negative and the left-hand side is positive,
which is a contradiction. So $t_0=1$ and $q=3q'+1$. On the other
hand, $q'^k-1=q$, i.e. $p_o^{k\beta}-p^{\alpha}=1$. Since $p'\geq
5$, so $k\geq 2$ and by Lemma \ref{cres1}, $\alpha=1$ or $q'=3$ and
$q=8$.

If $q'=3$ and $q=8$, then we get a contradiction since $d=1$.

If $\alpha=1$, then $p_o^{k\beta}=p+1$ and $p=3q'+1$. Since
$p=q\neq2$, thus $p_0=2$. Therefore $p=3\cdot 2^{\beta}+1$ and
$q'^k=3\cdot 2^{\beta}+2$. If $\beta>1$, then $q'^k=2(3\cdot
2^{\beta-1}+1)$, which is a contradiction. So $\beta=1$, $q'=2$,
$p'=7$ and $q=7$. Hence $11\in\pi(K/H)\backslash \pi(G)$, which is a
contradiction.

If $q\mid(q'^k+1)$, then similarly to the above we get a
contradiction.\\
{\bf{(b)}} Let $d'=(p',q'+1)=p'$.\\
$\blacktriangleright$ First suppose that $p=p_0$. Then
\begin{equation}\label{zb2}
(p^{3\alpha}+1)/(p^{\alpha}+1)=(p^{p'\beta}+1)/((p^{\beta}+1)p').
\end{equation}
Let $x$ be a primitive prime divisor of $(p^{2p'\beta}-1)$. So by
(\ref{zb2}), $x\mid (p^{3\alpha}+1)$ and thus $x\mid
(p^{6\alpha}-1)$. Therefore $p'\beta\mid 3\alpha$. Similarly we
conclude that $3\alpha \mid p'\beta$ and hence $3\alpha= p'\beta$.
Since $p'\geq 5$, so $\beta\leq \alpha$ and $q'\mid q$. Note that by
(\ref{abboz}), $(q'^{p'-1}-1)/(q'+1)=(p'q(q-1)+p'-1)/q'$. Then
$q'\mid (p'-1)$. On the other hand, $p'\mid (q'+1)$. So $q'=p'-1$,
$q'=2^{\beta}$ and $q=2^{\alpha}$. Thus by (\ref{zb2}),
$2^{2\alpha}-2^{\alpha}+1=(2^{p'\beta}+1)/(2^{\beta}+1)^2$, which
implies that
$$2^{p'\beta}=2^{2\alpha+2\beta}-2^{2\beta+\alpha}+2^{2\beta}+2^{2\alpha+\beta+1}-2^{\alpha+\beta+1}+2^{\beta+1}+2^{2\alpha}-2^{\alpha}.$$
If $\beta=1$, then $\alpha=3$ and $p'=9$, which is impossible.

Hence $\beta\neq1$ and so $\alpha=\beta+1$. But this is a
contradiction by the above equation.\\
$\blacktriangleright$ Therefore $p\neq p_0$. Note that by
\cite[Corollary 3]{spectra},
$\{(q'^{p'-1}-1)/p',(q'^{p'-2}-1)/p',(q'^{p'-3}-1)/p'\}\subseteq
\pi_{e}(K/H)$. So by Lemma \ref{omegakh}, $(q'^{p'-b}-1)/p'\mid
|\psu|$, for $b=1,2,3$. Since $(q'^{p'}+1)/((q'+1)p')=q^2-q+1$ and
$|\psu|=q^3 (q-1)(q+1)^2(q^2-q+1)$, hence $(q'^{p'-b}-1)/p'\mid
q^3(q-1)(q+1)^2$, where $b\in\{1,2,3\}$. Also
\begin{equation}\label{zb3}
q(q-1)=\frac{q'^{p'}-p'q'-p'+1}{p'(q'+1)}=\frac{q'^{p'-1}-q'^{p'-2}+q'^{p'-3}-\cdots+1-p'}{p'}.
\end{equation}
Therefore we can easily check that $((q'^{p'-1}-1)/p',q(q-1))\mid
(q'+1)(p'-1)/p'$, $((q'^{p'-2}-1)/p',q(q-1))\mid
(q'^2-p'q'+1-p')/p'$ and $((q'^{p'-3}-1)/p',q(q-1))\mid
(q'^3-p'q'+1-p')/p'$.

On the other hand, $((q'^{p'-1}-1)/p',(q'^{p'-2}-1)/p')\mid (q'-1)$,
$((q'^{p'-1}-1)/p',(q'^{p'-3}-1)/p')\mid (q'^2-1)$ and
$((q'^{p'-2}-1)/p',(q'^{p'-3}-1)/p')\mid (q'-1)$, by Lemma
\ref{bmm}. Consequently by Lemma \ref{xyu}, we get that
\begin{eqnarray*}
 A:=(\frac{q'^{p'-1}-1}{p' ((p'-1)(q'+1)/p')
(q'-1)(q'^2-1)}\times & \frac{q'^{p'-2}-1}{p'(
(q'^2-p'q'+1-p')/p')(q'-1)} \\
  \times \frac{q'^{p'-3}-1}{p' (
(q'^3-p'q'+1-p')/p')})\leq q^2(q+1)^2. &
\end{eqnarray*}
We claim that for $p'\geq 17$, $3q^2(q-1)^2\leq A$ and since $A\leq
q^2(q+1)^2$, we get that $2q^2-8q+2\leq 0$, which is a
contradiction, since $q\geq 4$.

Now we prove the claim. We know that
$q(q-1)=(q'^{p'}-p'q'-p'+1)/(p'(q'+1))\leq q'^{p'}/(p'(q'+1))$.
Hence $3q^2(q-1)^2\leq 3q'^{2p'}/(p'^2(q'+1)^2)\leq
3q'^{2p'-2}/p'^2$. In addition, easily we can see that
$q'^{p'-1}/((p'-1)q'^5)\leq (q'^{p'-1}-1)/((p'-1)(q'+1)^2(q'-1)^2)$,
$(q'^{p'-2}-1)/((q'-1)q'^2)\leq
(q'^{p'-2}-1)/((q'-1)(q'^2-p'q'+1-p'))$ and $q'^{p'-3}/q'^5\leq
(q'^{p'-3}-1)/(q'^3-p'q'+1-p').$ Consequently, $(q'^{p'-6}\times
q'^{p'-5}\times q'^{p'-8})/(p'-1)=q'^{3p'-19}/(p'-1)\leq A\leq
q^2(q+1)^2$. If $p'\geq 17$, then $q'^{2p'-2}/(p'-1)\leq
q'^{3p'-19}/(p'-1)$. Therefore by the above discussion and
(\ref{zb3}), we have
$$3q^2(q-1)^2\leq 3q'^{2p'-2}/p'^2\leq 3q'^{2p'-2}/(p'-1)^2\leq
q'^{2p'-2}/(p'-1)\leq q'^{3p'-19}/(p'-1)\leq A\leq q^2(q+1)^2$$
which is a contradiction. Thus $p'\in\{5,7,11,13\}$.

First suppose that $q>q'$. We know that
$(q'^{p'}+1)/((q'+1)p')=q^2-q+1$. Subtracting 1 from both sides, we
get that $(q'^{p'-1}-1)/(q'+1)=(p'q(q-1)+p'-1)/q'$. Put
$f=(p'q(q-1)+p'-1)_{p}$ and since $p\neq p_0$, so
$f=((p'q(q-1)+p'-1)/q')_p$. In addition, $q>q'$ and $p'\mid (q'+1)$.
So $p'-1\leq q'<q$ and hence $q\nmid (p'-1)$. Therefore
$f=(p'-1)_p$. Now by \cite{spectra}, ${\rm{PSU}}_{p'}(q')$ contains
an element of order $(q'^{p'-1}-1)/p'$. Since $p'\mid (q'+1)$, so
${\rm{PSU}}_{p'}(q')$ contains an element of order
$(q'^{p'-1}-1)/(q'+1)$. As we mentioned in the proof of Lemma
\ref{omegakh}, $\psu$ contains an abelian semisimple subgroup of
order $(p'q(q-1)+p'-1)/(fq')$. Now by Lemma \ref{malle},
$(p'q(q-1)+p'-1)/(fq')$ divides the order of a maximal torus $T$ of
$\psu$, which implies that $(p'q(q-1)+p'-1)/(fq')\mid
((p'q(q-1)+p'-1), |T|)$. Using Lemma \ref{torus} and easy
computation, we get that $((p'q(q-1)+p'-1), q^2+2q+1)\leq
9p'^2-6p'+1$, $((p'q(q-1)+p'-1), q^2-1)\leq 6p'-2$ and
$((p'q(q-1)+p'-1), q^2-q+1)=1$. Therefore $(p'q(q-1)+p'-1)/(fq')\leq
9p'^2-6p'+1$. Since $f=(p'-1)_p$, so $f\leq (p'-1)$ and $q>q'$
implies that $p'q(q-1)+p'-1<(9p'^2-6p'+1)(p'-1)q$.

Let $p'=5$. Hence $5q^2-5q+4<784q$. Thus $q<157$ is a prime power.
On the other hand by Lemma \ref{=ordercom},
$(q'^5+1)/(q'+1)=5(q^2-q+1)$. Now by GAP, we can see that the last
equation has no solution for prime powers $q<157$, which is a
contradiction.

If $p'\in\{7,11,13\}$, then we get a contradiction similarly.

Now suppose that $q'\geq q$. Let $p'=5$. Therefore by Lemma
\ref{=ordercom}, $(q'^{4}-1)/(q'+1)=(5q(q-1)+4)/q'$, which implies
that $q'^4-q'^3+q'^2-q'=5q^2-5q+4$. But this is a contradiction,
since $q'\geq q$. For $p'\in\{7,11,13\}$, we get a contradiction
similarly.

Also if $d=3$, then by a similar manner we get a contradiction.

If $K/H$ is isomorphic to ${\rm{PSL}}_{p'}(q')$, where $(p',q')\neq
(3,2),(3,4)$, then we get a contradiction similarly and for
convenience we omit the details of the
proof.\\
{\bf Case~3.}  Let $K/H\cong {\rm{PSU}}_{p'+1}(q')$, where $(q'+1)\mid (p'+1)$ and $q'=p_0^{\beta}$.\\
Let $d=1$. By Lemma \ref{=ordercom}, $(q'^{p'}+1)/(q'+1)=q^2-q+1$.
Thus
\begin{equation}\label{zb4}
q'(q'^{p'-1}-1)/(q'+1)=q(q-1).
\end{equation}
Therefore $q\mid q'$ or $q\mid (q'^{p'-1}-1)/(q'+1)$.\\
$\bullet$ Let $q\mid q'$. Therefore by (\ref{zb4}), $q=q'$ and
$(q^{p'-1}-1)/(q+1)=q-1$. Hence $p'=3$ and since $(q+1)\mid
(p'+1)$, so $q=3$, which is a contradiction, since $q$ is not a Fermat prime.\\
$\bullet$ Let $q\mid (q'^{p'-1}-1)/(q'+1)$. Now by \cite[Lemma
1.2]{vasil}, $(q'^{p'-1}-1)/(q'+1)$ divides the order of some
maximal torus of ${\rm{PSU}}_{p'+1}(q')$. Consequently, $p$ is
adjacent to every prime divisor of $(q'^{p'-1}-1)/(q'+1)$ in
$\Gamma(K/H)$. On the other hand by Remark \ref{2a2graph}, $p$ is
only adjacent to every primitive prime divisor of $q^2-1$ in
$\Gamma(G)$. Also by (\ref{zb4}), every prime divisor of
$(q'^{p'-1}-1)/(q'+1)$, is a prime divisor of $q-1$ or equals to
$p$. Hence $(q'^{p'-1}-1)/(q'+1)=q\cdot 2^t$, for some
$t\in\Bbb{N}\cup \{0\}$. By (\ref{zb4}), $q'\cdot 2^t=q-1$.

First suppose that $q$ is even. So $q'=q-1$, i.e.
$p_0^{\beta}=2^{\alpha}-1$, which implies that $q'=p_0=q-1$. Now
(\ref{zb4}) implies that $(q-1)^{p'-1}=q^2+1$, which is impossible.

Now let $q$ be odd and $q'$ be even. Therefore by
$(q'^{p'-1}-1)/(q'+1)=q\cdot 2^t$, $t=0$ and $q'=q-1$ by
(\ref{zb4}). Thus $q'^{p'-1}-1=(q'+1)^2$, i.e.
$2^{\beta(p'-1)}=2^{2\beta}+2^{\beta+1}+2$. So $\beta=0$, which is a
contradiction.

Let $q$ and $q'$ be odd. Since $(q'^{p'-1}-1)/(q'+1)=q\cdot 2^t$, so
by Lemma \ref{zsig}, we have $p'=3$ and $q'$ is not a Mersenne prime
or $p'=5$ and $q'$ is a Mersenne prime.

If $p'=3$ and $q'$ is not a Mersenne prime, then by (\ref{zb4}),
$q'(q'-1)=q(q-1)$. Therefore $q=q'$ and we get a contradiction as we
mentioned above.

If $p'=5$ and $q'$ is a Mersenne prime, then since $(q'+1)\mid
(p'+1)$, so $q'=2$ or $q'=5$. But in each case we get a
contradiction by (\ref{zb4}).

For $d=3$, we get the result similarly.

If $K/H \cong {\rm{PSL}}_{p'+1}(q')$, where $(q'-1)\mid (p'+1)$,
then we get a contradiction similarly.\\
{\bf Case~4.}  Let $K/H\cong {}^2G_{2}(q')$, where $q'=3^{2m+1}>3$.\\
The orders of the odd components of $K/H$ are $q'-\sqrt{3q'}+1$ and
$q'+\sqrt{3q'}+1$.

Let $d=1$. By Lemma \ref{=ordercom}, $q^2-q+1=3^{2m+1}\pm
3^{m+1}+1$. First suppose that $q^2-q+1=3^{2m+1}-3^{m+1}+1$.
Therefore
\begin{equation}\label{zb5}
q(q-1)=3^{m+1}(3^{m}-1).
\end{equation}
Therefore by (\ref{zb5}), $3^{m+1}\mid q$ or $3^{m+1}\mid (q-1)$. If
$3^{m+1}\mid (q-1)$, then $q\mid (3^m-1)$. Thus $3^{m+1}\leq
(q-1)\leq 3^m-2$, which is a contradiction. If $3^{m+1}\mid q$, then
by (\ref{zb5}), $(q-1)\mid (3^m-1)$. So $3^{m+1}\leq q\leq 3^{m}$,
which is impossible.\\
If $q^2-q+1=3^{2m+1}+3^{m+1}+1$, then similarly we get a
contradiction.

Now let $d=3$. So by Lemma \ref{=ordercom}, suppose that
$(q^2-q+1)/3=3^{2m+1}+3^{m+1}+1$. Hence
\begin{equation}\label{zb6}
(q-2)(q+1)=3^{m+2}(3^m+1).
\end{equation}
Since $(q-2,q+1)=3$, so by (\ref{zb6}), $3^{m+1}\mid (q+1)$ or
$3^{m+1}\mid (q-2)$.

If $3^{m+1}\mid (q+1)$, then by (\ref{zb6}), $(q-2)\mid (3(3^m+1))$.
Thus $q-5\leq 3^{m+1}$. Since $3^{m+1}\mid (q+1)$, so $q+1=3^{m+1}$
or $q+1=2\cdot 3^{m+1}$. If $q+1=3^{m+1}$, then by Lemma
\ref{cres1}, $q=8$. If $q+1=2\cdot 3^{m+1}$, then $2(q-5)\leq 2\cdot
3^{m+1}= q+1$. Therefore $q\in\{8,11\}$. If $q=8$, then we get a
contradiction by (\ref{zb6}). Let $q=11$. So by (\ref{zb6}), $m=1$.
But this is a contradiction since $3^{m+1}\mid (q+1)$.

If $3^{m+1}\mid (q-2)$, then by (\ref{zb6}), $(q+1)\mid 3(3^m+1)$.
Therefore $q-2=3^{m+1}$ and so $(q-2)^2=3q'$. Now by \cite[XI,
Theorems 13.2 and 13.4]{huppert3}, $q'-1\in \pi_e({}^2G_2(q'))$.
Consequently by Lemma \ref{omegakh}, $q'-1=(q^2-4q+1)/3\mid
|\psu|=q^3(q^2-1)(q^3+1)/3$. Now easily we can see that
$((q^2-4q+1)/3,q^3)=1$, $((q^2-4q+1)/3,(q^2-1)/3)\mid 2$,
$((q^2-4q+1)/3,q^2-q+1)\mid 3$ and $((q^2-4q+1)/3,q+1)\mid 6$.
Therefore $(q^2-4q+1)/3\mid 36$. Thus $q=5$, which is a
contradiction since $q$ is not a Fermat prime.

If $K/H$ is isomorphic to ${}^2F_4(2^{2m+1})$, where $2^{2m+1}>2$,
${}^2B_2(2^{2m+1})$, where $2^{2m+1}>2$ or $G_2(q')$, where $2<q'$,
then
we get a contradiction similarly.\\
{\bf Case~5.} Let $K/H\cong B_n(q')$, where $n=2^m\geq 4$ and $q'$
is odd.

Let $d=1$. So $(q'^n+1)/2=q^2-q+1$. Therefore $q(q-1)=(q'^n-1)/2$.
By \cite[Lemma 1.2]{vasil}, $(q'^n-1)/2$ is the order of a maximal
torus of $B_n(q')$. Hence $p$ is adjacent to every prime divisor of
$q-1$ in $\Gamma(G)$. But this is a contradiction since $p\nsim r_1$
in $\Gamma(G)$ by Remark \ref{2a2graph}.

Let $d=3$. Thus $(q'^n+1)/2=(q^2-q+1)/3$. So
$(q'^n-1)/2=(q^2-q-2)/3=(q-2)(q+1)/3$. Now by \cite[Corollary
5]{spectrao}, $B_n(q')$ contains some element of order $(q'^n-1)/2$.
Therefore by Lemma \ref{omegakh}, $(q-2)(q+1)/3\mid
|\psu|=q^3(q-1)(q+1)^2(q^2-q+1)/3$. Since $((q-2)/3,
(q^2-q+1)/3)=1$, so $(q-2)/3\mid q^3(q-1)(q+1)$. Now it is
straightforward to see that $((q-2)/3,q)\mid 2$, $((q-2)/3,q-1)=1$
and $((q-2)/3,q+1)\mid 3$. Consequently $(q-2)/3\mid 24$ and thus
$q\in\{5,8,11\}$. Since $q$ is not a Fermat prime, so $q\neq 5$. If
$q=8$, then by $(q'^n+1)/2=(q^2-q+1)/3$, $q'=37$ and $n=1$, which is
a contradiction, since $n$ is even. For $q=11$, we get a
contradiction similarly.

If $K/H$ is isomorphic to $C_{n}(q')$, where $n=2^m\geq 2$  or ${}^2D_n(q')$, where $n=2^m\geq 4$, then we get a contradiction similarly.\\
{\bf Case~6.} Let $K/H\cong A_1(q')$, where $3<q'\equiv \epsilon
\pmod{4}$, $\epsilon=\pm 1$. So the orders of odd components of
$A_1(q')$ are $q'$ and $(q'+\epsilon)/2$.

First suppose that $d=1$. So  by Lemma \ref{=ordercom},
$q^2-q+1\in\{q',(q'-1)/2,(q'+1)/2\}$.

If $q^2-q+1=(q'-1)/2$, then $q^2-q+2=(q'+1)/2$. Now by
\cite[Corollary 3]{spectra}, $A_1(q')$ contains some element of
order $(q'+1)/2$. Thus by Lemma \ref{omegakh}, $(q^2-q+2)\mid
|\psu|=q^3(q-1)(q+1)^2(q^2-q+1)$. By easy computation, we get that
$(q^2-q+2,q)\mid 2$, $(q^2-q+2,q-1)\mid 2$, $(q^2-q+2,q+1)\mid 4$,
$(q^2-q+2,q^2-q+1)=1$. Therefore $(q^2-q+2)\mid 256$. Therefore we
obtain that $q\in\{2,3\}$, which is a contradiction.

If $q^2-q+1=q'$, then $q(q-1)=q'-1$. By \cite[Lemma 1.2]{vasil},
$q'-1$ divides the order of some maximal torus in $A_1(q')$.
Therefore $p\sim r_1$ in $\Gamma(G)$, which is a contradiction by
Remark \ref{2a2graph}.

If $q^2-q+1=(q'+1)/2$, then similarly we get a contradiction.

Let $d=3$. Suppose that  $(q^2-q+1)/3=(q'-1)/2$. So
$(q^2-q+4)/3=(q'+1)/2$. Since by \cite[Corollary 3]{spectra},
$(q'+1)/2\in \pi_e(A_1(q'))$, so similarly to the above we get that
$(q^2-q+4)/3\mid 3^2\cdot 2^{10}$. Thus $q=4$, which is a
contradiction since $d=3$.

If $(q^2-q+1)/3=q'$, then $(q^2-q-2)/6=(q-2)(q+1)/6=(q'-1)/2$. So by
\cite[Corollary 3]{spectra}, $(q'-1)/2\in \pi_e(A_1(q'))$. Similarly
to the above, we obtain that $(q-2)/6\mid 24$. Consequently $q=8$
and $q'=19$. Now $5\in\pi(A_1(19))\backslash \pi({\rm{PSU}}_3(8))$,
which is a contradiction.

If $(q^2-q+1)/3=(q'+1)/2$, then similarly we get a contradiction.

If $K/H\cong A_1(q')$, where $q'>2$ is even, then we get a contradiction similarly.\\
{\bf Case~7.} Let $K/H\cong E_{6}(q')$, where $q'=p_0^{\beta}$.
Therefore $(q^2-q+1)/d=(q'^6+q'^3+1)/(3,q'-1)$.

Let $d=(3,q'-1)$. Thus $q'^3(q'^3+1)=q(q-1)$. By \cite[Lemma
1.2]{vasil}, $q'^3+1$ divides the order of some maximal torus of
$E_6(q')$. On the other hand, by \cite[Table~5]{vasil}, $p_0$ is
adjacent to every prime divisor of $q'^3+1$. Consequently, $p\sim
r_1$, which is a contradiction by Remark \ref{2a2graph}.

Let $d=1$ and $(3,q'-1)=3$. Hence $q'^3(q'^3+1)=3q^2-3q+2$. Since by
\cite{atlas}, $\pi(q'^3(q'^3+1))\subseteq
\pi_1(K/H)\subseteq\pi_1(G)=\pi(q)\cup\pi(q-1)\cup\pi(q+1)$.
Therefore if $p''$ is a prime divisor of $3q^2-3q+2$, then $p''=2$.
So $\pi(q'^3(q'^3+1))=\{2\}$, which is a contradiction.

Let $d=3$ and $(3,q'-1)=1$. Thus $(q^2-q+1)/3=q'^6+q'^3+1$. Put
$r:=q'^3$. Therefore $(q+1)(q-2)=3r(r+1)$. Since $(q+1,q-2)=3$, so
$r\mid (q+1)$ or $r\mid (q-2)$. If $r\mid (q+1)$, then $q+1=rw$, for
some $w\in \Bbb{N}$. Hence $r=3(w+1)/(w^2-3)$. For $w\geq 3$, we
have $r\leq 2$, which is impossible. If $w=2$, then $r=9$, which is
a contradiction. For $w=1$, $r=-3$, which is impossible. If $r\mid
(q-2)$, then we get a contradiction similarly.

If $K/H\cong {}^2E_6(q')$, where $q'>2$, then we get a contradiction
similarly.\\
{\bf Case~8.} Let $K/H\cong E_{8}(q')$, where $q'=p_0^{\beta}$.
Therefore
$(q^2-q+1)/d\in\{q'^8-q'^4+1,(q'^{10}+q'^5+1)/(q'^2+q'+1),(q'^{10}-q'^5+1)/(q'^2-q'+1),(q'^{10}+1)/(q'^2+1)\}$.\\
$\bullet$ Let $(q^2-q+1)/d=q'^8-q'^4+1$. Suppose that $d=1$.
Therefore $q'^4(q'^4-1)=q(q-1)$. By \cite[Lemma 1.3]{vasil},
$q'^4-1$ divides the order of some maximal torus of $E_8(q')$. On
the other hand by \cite{vasil}, $p_0$
 is adjacent to every prime divisor of $q'^4-1$ in $\Gamma(K/H)$. Hence $p\sim r_1$, which is a contradiction by Remark
 \ref{2a2graph}.

Therefore $d=3$ and $(q^2-q+1)/3=q'^8-q'^4+1$. Then
$(q-2)(q+1)=3q'^4(q'^4-1)$. Since $(q-2,q+1)=3$, so $q'^4\mid (q-2)$
or $q'^4\mid (q+1)$. If $q'^4\mid (q-2)$, then $q-2=q'^4t$, for some
$t\in\Bbb{N}$. Hence $q'^4(t^2-3)=-3(t+1)$. Consequently $t=1$ and
$q'^4=3$, which is impossible. If $q'^4\mid (q+1)$, then
$q+1=q'^4s$, for some $s\in\Bbb{N}$. Therefore
$q'^4=3(s-1)/(s^2-3)$. It is clear that $s\geq 2$ and so
$(s-1)/(s^2-3)\leq 1$. Thus $q'^4\leq 3$, which is a
contradiction.\\
$\bullet$ Now suppose that $(q^2-q+1)/d=(q'^{10}+1)/(q'^2+1)$.

Suppose that $d=1$. Therefore
\begin{equation}\label{zb7}
q(q-1)=q'^2(q'^4+1)(q'^2-1).
\end{equation}
If $p=p_0$, then $q=q'^2$ and $q'^4+1=1$, which is impossible. Thus
$p\neq p_0$. Let $q$ be even. Then $q'$ is odd and
$(q'^4+1,q'^2-1)=2$. Hence $q\mid 2(q'^2-1)$ or $q\mid 2(q'^4+1)$.

If $q\mid 2(q'^2-1)$, then $q'^2(q'^4+1)/2\leq q-1\leq 2(q'^2-1)-1$,
which is a contradiction.

If $q\mid 2(q'^4+1)$, then $2(q'^4+1)=qu$, for some $u\in\Bbb{N}$.
Therefore by (\ref{zb7}), $q'^2((4-u^2)q'^2+u^2)=2u-4$. For
$u\in\{1,2\}$, the last equation has no solution. Hence $u\geq 3$
and $$63\leq 9(9+u^2/(4-u^2))\leq
q'^2(q'^2+u^2/(4-u^2))=(2u-4)/(4-u^2)=-2/(2+u)<1,$$ which is a
contradiction. If $q$ is odd, then similarly we get a contradiction.

Therefore $d=3$. So
\begin{equation}\label{zb8}
q^2-q-2=(q+1)(q-2)=3q'^2(q'^4+1)(q'^2-1).
\end{equation}
Then $q$ is odd, because otherwise $(q^2-q-2)_2=2$, which implies
that $q'$ is odd by (\ref{zb8}). Consequently $8\mid (q'^2-1)$,
which is a contradiction. We know that $A_8(q')< E_8(q')$, by
\cite{embed}. By \cite{spectra}, $A_8(q')$ contains some element of
order $(q'^8-1)/(9,q'-1)$ and so some element of order
$(q'^4+1)(q'^2-1)/(9,q'-1)$. Let
$k=((q-2)/3,(q'^4+1)(q'^2-1)/(9,q'-1))$. Then $E_8(q')$ contains
some semisimple element of order $k$. Thus by Lemma \ref{omegakh},
$k\mid |\psu|$. It is easy to verify that $((q-2)/3,q-1)=1$,
$((q-2)/3,q+1)\mid 3$, $((q-2)/3,q)=1$ and
$((q-2)/3,(q^2-q+1)/3)=1$. Hence $k\in\{1,3,9\}$.

If $k=1$, then $(q'^4+1)(q'^2-1)\mid (q+1)$ by (\ref{zb8}). Thus
$q+1=v(q'^4+1)(q'^2-1)$, for some $v\in\Bbb{N}$. So by (\ref{zb8}),
$(q'^4-q'^2+1)v^2q'^2=3q'^2+3v+v^2$. For $v\in\{1,2\}$, the last
equation has no solution. Hence $v\geq 3$. Therefore
$v^2q'^2=(3q'^2+3v+v^2)/(q'^4-q'^2+1)<1+(3v+v^2)/(q'^4-q'^2+1)\leq
1+2v^2/(q'^4-q'^2+1)$, which is a contradiction. If $k\in\{3,9\}$,
then we get a contradiction similarly.\\
$\bullet$ Let $(q^2-q+1)/d=(q'^{10}-q'^5+1)/(q'^2-q'+1)$.

Suppose that $d=1$. Then
\begin{equation}\label{zb9}
q(q-1)=q'(q'^4-1)(q'^3+q'^2-1).
\end{equation}
If $(q,q')\neq 1$, then $q=q'$ by (\ref{zb9}), which is a
contradiction by (\ref{zb9}). Thus $(q,q')=1$. Since
$(q'^3+q'^2-1,q'^4-1)$ divides $5$, so $q\mid (q'^4-1)(5,q)$ or
$q\mid (q'^3+q'^2-1)(5,q)$. Suppose that $q\mid (q'^4-1)(5,q)$. If
$q$ is odd, then $q\mid 5(q'^2+1)$ or $q\mid 5(q'^2-1)$.

If $q\mid 5(q'^2+1)$, then $q'(q'^2-1)(q'^3+q'^2-1)/5\leq (q-1)$, by
(\ref{zb9}). Thus $q'(q'^2-1)(q'^3+q'^2-1)/5\leq 5(q'^2+1)-1$, which
is a contradiction.

If $q\mid 5(q'^2-1)$ or $q$ is even, then we get a contradiction
similarly.

Therefore $q\mid (q'^3+q'^2-1)(5,q)$. Hence $q'(q'^4-1)/(5,q)\mid
(q-1)$ and $q-1\leq (5,q)(q'^3+q'^2-1)-1$. If $(5,q)=1$, then
$q'(q'^4-1)\leq q'^3+q'^2-2$, which is a contradiction. Therefore
$(5,q)=5$ and $q'(q'^4-1)\leq 25(q'^3+q'^2-1)-5$. It is easy to
verify that $q'\leq 5$. Since $(5,q)=5$ and $(q,q')=1$, so $q'=4$.
Now by (\ref{zb9}) we get a contradiction.

Therefore $d=3$. Then
\begin{equation}\label{zb10}
(q+1)(q-2)=3q'(q'^4-1)(q'^3+q'^2-1).
\end{equation}
Since $A_1(q')\circ A_5(q')<A_8(q')<E_8(q')$, where $\circ$ denotes
the central product of groups, so \cite{spectra} implies that
$E_8(q')$ contains some element of order $p_0(q'^4-1)$. Put
$m=((q-2)/3,p_0(q'^4-1))$. Similarly to the above, we get that
$m\in\{1,3,9\}$.

Suppose that $p_0$ divides $m$. Then $p_0=3$, $q'\mid (q-2)$ and
$(q'^4-1)\mid (q+1)$, by (\ref{zb10}). So $q+1=(q'^4-1)w'$, for some
$w'\in\Bbb{N}$. Therefore by (\ref{zb10}), we have
$q'^4(w'^2-3)=w'^2+3w'+3q'^3-3q'$. If $w'=1$, then
 $$q'=-2/q'^3-3/2+3/(2q'^2)<1/q'^3+3/(2q'^2)<2,$$ which is a contradiction. Thus $w'>1$ and consequently
$$q'=(w'^2+3w')/q'^3(w'^2-3)+3q'(q'^2-1)/q'^3(w'^2-3)<1+3(w'+1)/(w'^2-3)+3/(w'^2-3).$$Since $1+3(w'+1)/(w'^2-3)+3/(w'^2-3)$
is equal to $13$, for $w'=2$ and is at most $10$, for $w'\geq 3$, so
$q'\in\{2,3,4,5,7,8,9\}$. But in each case, equation (\ref{zb10})
has no solution.

Therefore $p_0$ does not divide $m$. Consequently by (\ref{zb10}),
$q'\mid (q+1)$ and $(q'^4-1)\mid (q+1)$. So $q+1=q'(q'^4-1)y/m$, for
some $y\in\Bbb{N}$. Using (\ref{zb10}), we have
$(q'(q'^4-1)y/m)\times (q'(q'^4-1)y/m-3)=3q'(q'^4-1)(q'^3+q'^2-1)$.
Hence $q'^5=3q'^3m^2/y^2+3q'^2m^2/y^2+q'+3m(y-m)/y^2$. In
particular, $q'\mid 3(y-m)$.

Let $y\geq m$. Then $$q'^2\leq 3+3/q'+1/q'^2+3/q'^3\leq 3+3/2+1/4+3/8<6,$$ which implies that $q'=2$ and $q=32$.
 But this is a contradiction, since $61\in\pi(E_8(2))\backslash \pi({\rm{PSU}}_3(32))$ by \cite{atlas}.

Thus $m>y$ and so $m\neq 1$. Since $q'\mid 3(y-m)$, $q'\leq 24$.
Hence $q'\in\{2,3,4,5,7,8,9\}$. But in each case, equation
(\ref{zb10}) has no solution, which is a contradiction.

If $(q^2-q+1)/d=(q'^{10}+q'^5+1)/(q'^2+q'+1)$, then we get a
contradiction similarly.

If $K/H$ is isomorphic to $F_4(q')$ or ${}^3D_4(q')$, then we get a
contradiction similarly.\\
{\bf Case~9.} Let $K/H\cong C_{p'}(q')$, where $q'=2,3$, or
$B_{p'}(3)$, or $D_{p'+1}(q')$, where $q'=2,3$, or ${}^2D_n(2)$,
 where $n=2^m+1\geq 5$, or ${}^2D_{p'}(3)$,
 where $5\leq p'\neq 2^m+1$, or ${}^2D_{p'}(3)$, where $p'=2^m+1$,
 or $D_{p'}(q')$, where $p'\geq 5$ and $q'=2,3,5$.

We discuss only the case $K/H\cong C_{p'}(q')$, where $q'=2,3$,
because the other cases are similar. Lemma \ref{=ordercom} implies
that $(q'^{p'}-1)/(q'-1)=(q^2-q+1)/d$.

Suppose that $q'=2$. Then $2^{p'}-1=(q^2-q+1)/d$. Let $d=1$.
Subtracting 1 from both sides of this equality, we have
$2(2^{p'-1}-1)=q(q-1)$. By \cite[Lemma 1.2]{vasil}, $(2^{p'-1}-1)$
is the order of a maximal torus of $C_{p'}(q')$. Also by
\cite[Table~4]{vasil}, $q'=2$ is adjacent to all elements of
$\pi(K/H)$ except $u_{p'}$ and $u_{2p'}$ in $\Gamma(K/H)$, where
$u_i$ denotes a primitive prime divisor of $q'^i-1$. Hence $2$ is
adjacent to every prime divisor of $(2^{p'-1}-1)$ in $\Gamma(K/H)$.
Consequently $p\sim r_1$ in $\Gamma(G)$, which is a contradiction by
Remark \ref{2a2graph}.

Let $d=3$. Thus $2^{p'}-1=(q^2-q+1)/3$ and so
$2^{p'-1}-1=(q+1)(q-2)/6$. By \cite{spectrao}, $C_{p'}(q')$ contains
some element of order $2^{p'-1}-1$. Hence by Lemma \ref{omegakh},
$(q+1)(q-2)/6\mid |\psu|$. Now we get a contradiction as Case~5.

For $q'=3$, we get the result similarly.\\
{\bf Case~10.} Let $K/H\cong {}^2D_{n}(3)$, where $n=2^m+1\geq 9$ is
not a prime number. So $(3^{n-1}+1)/2=(q^2-q+1)/d$.

Let $d=1$. Thus $3^{n}+1=6q^2-6q+4$. Since $\pi(3^n+1)\subseteq
\pi_1(K/H)\subseteq \pi_1(G)$. Hence if $a_1\in \pi(3^n+1)$, then
$a_1\in\pi(q)\cup \pi(q-1)\cup \pi(q+1)$, by Remark \ref{2a2graph}.

If $a_1\in\pi(q)$, then by the equation $3^{n}+1=6q^2-6q+4$,
$a_1=2$. If $a_1\in\pi(q+1)$, then since $3^n+1=6q^2+6q-12q+4$, so
$a_1\mid (12q-4)$. Also $a_1\mid (12q+12)$. Hence $a_1=2$. If
$a_1\in\pi(q-1)$, then we get that $a_1=2$. Consequently
$3^n+1=2^{\kappa}$, for some $\kappa \in \Bbb{N}$. Since $n\geq 9$,
so by Lemma \ref{cres2}, we get a contradiction.

Let $d=3$. Therefore $(3^{n-1}+1)/2=(q^2-q+1)/3$. Subtracting 1 from
both sides, we have $(3^{n-1}-1)/2=(q+1)(q-2)/3$. But this is a
contradiction, since 3 divides the right-hand side of the last
equation, while 3 does not divide the left-hand side of this
equation.\\
{\bf Case~11.} Let $K/H$ be isomorphic to $A_2(2)$, $A_2(4)$,
$^2A_3(2)$, $^2A_5(2)$, $Co_1$, $Co_{2}$, $Co_3$, $E_7(2)$,
$E_7(3)$, ${}^2E_6(2)$, $F_1$, $F_2$, $F_3$, $Fi_{22}$, $Fi_{23}$,
$Fi^{'}_{24}$, ${}^2F_4(2)'$, $He$, $HN$, $HiS$, $J_1$, $J_2$,
$J_3$, $J_4$, $LyS$, $M_{11}$, $M_{12}$, $M_{22}$, $M_{23}$,
$M_{24}$, $McL$, $O^{'}N$, $Ru$ or $Suz$. Since the proofs are
similar, so for convenience we give the proof only for $J_4$.

Let $K/H\cong J_4$. Therefore $(q^2-q+1)/d\in\{23,29,31,37,43\}$.
Let $(q^2-q+1)/d=23$. If $d=1$, then $q(q-1)=22$, which is
impossible. If $d=3$, then $q(q-1)=68$, which is a contradiction.

Let $(q^2-q+1)/d=37$. If $d=1$, then $q(q-1)=36$, which is
impossible. If $d=3$, then $q(q-1)=110$ and so $q=11$. Now
$43\in\pi(J_4)\backslash \pi({\rm {PSU}}_3(11))$, which is a
contradiction.

In other cases, we get a contradiction similarly.

Therefore $G$ has a unique nonabelian composition factor isomorphic
to $\psu$ and the proof is completed.
\end{proof}
\begin{theorem}\label{u39}
Suppose that $G$ is  a finite group such that $M(G)=M({\rm{PSU}}_3(9))$. Then
$G$ has a unique nonabelian composition factor which is isomorphic
to ${\rm{PSU}}_3(9)$.
\end{theorem}
\begin{proof}
In the sequel, we denote by $r_i$ a primitive prime divisor of $9^i-1$. Also we have $\Gamma(G)=\Gamma({\rm{PSU}}_3(9))$. By \cite{atlas}, $\pi({\rm{PSU}}_3(9))=\{2,3,5,73\}$, $|{\rm{PSU}}_3(9)|=2^5\cdot 3^6\cdot 5^2\cdot 73$. Also the prime graph of ${\rm{PSU}}_3(9)$ has two complete connected components, $\pi_1({\rm{PSU}}_3(9))=\{2,3,5\}$ and $\pi_2({\rm{PSU}}_3(9))=\{73\}$. First we prove that $G$ is neither a Frobenius group nor a
2-Frobenius group. Let $G=FC$ be a Frobenius group with
kernel $F$ and complement $C$. Since $F$ is nilpotent, so
$\Gamma(F)$ is complete. Also by Lemma
\ref{froprime}, $C$ is solvable, since $\Gamma(C)$ is complete.

Let $\pi(F)=\{73\}$ and $\pi(C)=\{2,3,5\}$. Since $73\| |{\rm{PSU}}_3(9)|$ and every group of order $73^2$ is abelian, so $|F|=73$. On the other hand by Lemma \ref{froprime}, $|C|\mid (|F|-1)=72$, which is a contradiction.

Let $\pi(F)=\{2,3,5\}$ and $\pi(C)=\{73\}$. Suppose that $P$ is a Sylow 5-subgroup of $G$. We know that $Z(P)$ char $P$ and $P$ char $F\unlhd G$. Therefore $Z(P)\unlhd G$. Hence $C$ acts on $Z(P)$ via conjugation. Since $5\nsim 73$ in $\Gamma(G)$, so $C$ acts fixed point freely on $Z(P)$. Thus there is a Frobenius group  with kernel $Z(P)$ and complement $C$. So by Lemma \ref{froprime}, $73\mid (|Z(P)|-1)$. In addition, $Z(P)$ is abelian and hence $|Z(P)|=5$ or $5^2$, which is a contradiction. Consequently $G$ is not a Frobenius group.

Let $G$ be a $2$-Frobenius group. Thus $G$ has a normal series $1\unlhd H\unlhd K\unlhd G$ such that $K$ is a Frobenius group with kernel $H$ and $G/H$ is a Frobenius group with kernel $K/H$. Also by Lemma \ref{ijac}, $\Gamma(G)$ has two components $\Gamma(K/H)$ and $\Gamma(G/K)\cup \Gamma(H)$. Since $73^2 \nmid |G|$, we get that $|K/H|=73$. Now similarly to the above, we have $K/H$ acts fixed point freely on $Z(P)$, where $P$ is a Sylow 5-subgroup of $K$ and it is a contradiction. Therefore $G$ is not a 2-Frobenius group.

Hence by Lemma \ref{ijac}, $G$ has a normal series $1\unlhd H\unlhd K\unlhd G$ such that $H$ and $G/K$ are $\pi_1$-groups, while $H$ is nilpotent and $K/H$ is a nonabelian simple group. In addition by Lemma \ref{=ordercom}, $s(K/H)\geq 2$ and the odd
order component of ${\rm{PSU}}_3(9)$ is an odd order component of $K/H$. Hence $K/H$ is either a $K_3$-simple group or a $K_4$-simple group. By \cite{atlas}, there are only eight $K_3$-simple groups, but 73 does not divide the order of any of them. By \cite[Theorem 1]{k4group}, $K/H$ can be isomorphic to ${\rm {PSL}}_3(2^3)$, ${\rm {PSU}}_3(3^2)$ or ${\rm {PSL}}_2(q')$, where $q'$ is a prime power satisfying $q'(q'^2-1)=(2,q'-1)2^{\alpha_1}3^{\alpha_2}p''^{\alpha_3}r^{\alpha_4}$, $\alpha_i\in \Bbb{N}, i=1\cdots 4$, with $p''>3$ and $r>3$ distinct prime numbers.

If $K/H\cong {\rm {PSL}}_3(2^3)$, then $73\in \pi(K/H)\backslash \pi(G)$, which is a contradiction.

If $K/H\cong {\rm {PSL}}_2(q')$, then the odd order components of ${\rm {PSL}}_2(q')$ are $q'$, $(q'\pm 1)/2$, $q'-1$ or $q'+1$.

Let $q'=73$. So $37\in \pi({\rm {PSL}}_2(73))\backslash \pi(G)$, which is a contradiction. If $(q'-1)/2=73$, then $q'=3\cdot 7^2$, which is impossible. In other cases, we get a contradiction.

Consequently, $K/H\cong {\rm{PSU}}_3(9)$.
\end{proof}
\begin{theorem}\label{dfg}
Suppose that $q=p^{\alpha}>2$ is not a Fermat prime. Then:\\
i) If $p_0\mid \alpha$ and $(p_0,6)=1$, then $M(\psu)\neq
M(\psu\cdot \langle \theta\rangle)$, where $\theta$ is a field automorphism of $\psu$ and $|\langle\theta\rangle|=p_0$.\\
ii) If $3\mid \alpha$ and $(q+1)_3=3$ or $3\mid (q-1)$, then
$M(\psu)\neq M(\psu\cdot \langle \phi\rangle)$, where $\phi$ is a
field automorphism of $\psu$ and $|\langle
\phi\rangle|=3$.\\
iii) If $q=3^{\alpha}$ and $3\mid \alpha$, where $\alpha\geq 9$,
then $M(\psu)\neq M(\psu\cdot \langle \phi\rangle)$, where $\phi$ is
a field
automorphism of $\psu$ and $|\langle \phi\rangle|=3$.\\
iv) If $q=2^{\alpha}$ and $2\mid \alpha$, then $M(\psu)\neq
M(\psu\cdot \langle \sigma\rangle)$, where $\sigma$ is a field
automorphism of $\psu$ and $|\langle \sigma\rangle|=2$.
\end{theorem}
\begin{proof}
(i) Let $\psu\cdot \langle\theta\rangle$ be the extension of $\psu$
by a field automorphism, where $\theta$ is of order $p_0$.

Now suppose that $q=q'^{p_0}$. Let $x$ be a primitive prime divisor
of ${q'}^{6}-1$. Thus $x$ is a primitive prime divisor of $q^6-1$,
since $(p_0,6)=1$. Hence $x\in \pi_2(\psu)$. Since $\theta$ is a
field automorphism, so by \cite[Proposition 4.9.1 and Table
4.5.1]{gorlysol} $C_{\psu}(\theta)={\rm{PSU}}_3(q')$ is a subfield
subgroup. Therefore $p_0\sim x$ and $p_0\sim p$ in
$\Gamma(\psu\cdot\langle\theta\rangle)$, which implies that
$\Gamma(\psu)\neq \Gamma(\psu\cdot \langle \theta\rangle)$, since
$p\in\pi_1(\psu)$ and $x\in \pi_2(\psu)$. Therefore $M(\psu)\neq
M(\psu\cdot \langle \theta\rangle)$.

(ii) Let $3\mid \alpha$. Put $q=q_0^{3}$ and note that $p\neq 3$.
Let $G=\psu\cdot \langle \phi\rangle$, be the extension of $\psu$ by
a field automorphism $\phi$ of order 3. Similarly to the above, we
have $C_{\psu}(\phi)={\rm {PSU}}_3(q_0)$. Therefore $3\sim p$ in
$\Gamma(G)$. On the other hand, $q-1\neq 2^k$ and so by
\cite[Table~4]{vasil} if $(q+1)_3=3$, then $p\nsim 3$ in
$\Gamma(\psu)$. Similarly if $3\mid (q-1)$, then
\cite[Table~4]{vasil} shows that $p\nsim r_1=3$. Therefore
$\Gamma(\psu)\neq \Gamma(\psu\cdot\langle \phi\rangle)$ and so
$M(\psu)\neq M(\psu\cdot\langle \phi\rangle)$.

(iii) Now let $p=3$ and $3\mid \alpha$, where $\alpha\geq 9$. Let
$G=\psu\cdot \langle \phi\rangle$, be the extension of $\psu$ by a
field automorphism $\phi$ of order 3. Similarly to the proof of (ii)
for $q=q_0^3$, we have $C_{\psu}(\phi)={\rm {PSU}}_3(q_0)$. If $y_0$
is a primitive prime divisor of $q_0-1$, then $y_0$ is a primitive
prime divisor of $q-1$, say $r_1$. Hence $3\sim r_1$ in
$\Gamma(\psu\cdot\langle \phi\rangle)$. On the other hand by
\cite[Table 4]{vasil}, $p=3$ is not adjacent to any primitive prime
divisor of $q-1$ except $r_1=2$ in $\Gamma(\psu)$. Consequently if
$\Gamma(\psu)=\Gamma(\psu\cdot\langle \phi\rangle)$, then every
prime divisor of $q_0-1$ is equal to 2. So $q_0-1=2^{\delta}$, for
some $\delta\in \Bbb{N}$. Now by Lemma \ref{cres1}, $q_0=3$ or
$q_0=9$ and so $q=3^3$ or $q=3^6$, which is a contradiction.
Therefore $M(\psu)\neq M(\psu\cdot \langle \phi\rangle)$, where
$|\langle \phi\rangle|=3$.

(iv) Note that \cite[Proposition 4.9.1]{gorlysol} asserts that all
automorphisms of prime order are conjugate. Let $2\mid \alpha$ and
$\sigma$ be a field automorphism of order 2. The centralizer of the
field automorphism of order 2 of $\psu$, is the same as the
centralizer of field and graph automorphisms of ${\rm{PSL}}_3(q^2)$,
which is the orthogonal group ${\rm{PSO}}_3(q)$.

Let $q$ be odd. Then by \cite[Table 4.5.1]{gorlysol}, the
centralizer of $\sigma$ is a 3-dimensional orthogonal group
${\rm{PSO}}_3(q)$ which is isomorphic to ${\rm {PGL}}_2(q)$. Now by
\cite{oliverking}, the orders of maximal abelian subgroups of ${\rm
{PGL}}_2(q)$ are $q$, $(q-1)$ and $(q+1)$. Thus the orders of
maximal abelian subgroups which may lie in $\psu\cdot
\langle\sigma\rangle\backslash \psu$ are $2q$, $2(q-1)$ or $2(q+1)$.
Also by \cite{spectra}, $\psu$ contains some elements of order
$(q^2-1)/d$. Let $d=1$. So $\psu$ contains some abelian cyclic
subgroup of order $q^2-1$. Since $q$ is odd, hence $2(q-1)\mid
(q^2-1)$ and $2(q+1)\mid (q^2-1)$. On the other hand, one of the
parabolic maximal subgroups of $\psu$ is isomorphic to ${\rm
{GU}}_2(q)$, which has an abelian subgroup of order $2q$. Therefore
$\psu$ contains some abelian subgroup of order $2q$.

If $d=3$, then $\psu$ contains some abelian subgroups of orders
$(q^2-1)/3$ and $2q$.

Hence if $q$ is odd, then it is possible that
$M(\psu)=M(\psu\cdot\langle\sigma\rangle)$, where $\sigma$ is a
field automorphism of order 2.

Let $q$ be even. So similarly to the above, 2 is adjacent to every
primitive prime divisor of $q-1$, which implies that $M(\psu)\neq
M(\psu\cdot\langle\sigma\rangle)$, where $\sigma$ is a field
automorphism of order 2, since by \cite[Table~4]{vasil}, $2\nsim
r_1$ in $\Gamma(\psu)$.
\end{proof}
\begin{theorem}\label{dfg2}
Let $G$ be a finite group such that $M(G)=M(\psu)$, where
$q=p^{\alpha}>2$ is not a Fermat prime. Then one of the following holds:\\
1) If $2\nmid \alpha$ and $3\nmid \alpha$, then $G\cong \psu$.\\
2) If $2\nmid \alpha$, $3\mid \alpha$ and $q=3^{\alpha}\neq
3^3$, then $G\cong \psu$.\\
3) If $2\mid \alpha$, $3\nmid \alpha$ and $q=2^{\alpha}$, then $G\cong \psu$.\\
4) If $2\nmid \alpha$, $3\mid \alpha$ and $(q+1)_3=3$, then $G\cong \psu$.\\
5) If $2\nmid \alpha$, $3\mid \alpha$ and $3\mid
(q-1)$, then $G\cong \psu$.\\
6) If $2\mid \alpha$, $3\nmid \alpha$ and $q$ is odd, then $G\cong
\psu$ or $G\cong \psu\cdot \langle \varphi\rangle$,
 the extension of $\psu$ by a field automorphism of $\psu$, where $|\langle \varphi\rangle|=2^i$ and $2\leq 2^i\leq (\alpha)_2$.\\
7) If $2\nmid \alpha$, $3\mid \alpha$ and $(q+1)_3> 3$, or $q=3^3$,
then $G\cong \psu$ or $G\cong \psu\cdot \langle \tau\rangle$, the
extension of $\psu$ by a field automorphism of $\psu$, where
$|\langle
\tau\rangle|=3^i$ and $3\leq 3^i\leq (\alpha)_3$.\\
8) If $6\mid \alpha$ and $q=2^{\alpha}$, then $G\cong \psu$ or
$G\cong \psu\cdot \langle \varphi\rangle$, the extension of $\psu$
 by a field automorphism of $\psu$, where $1\leq |\langle \varphi\rangle|=3^k\leq (\alpha)_3$.\\
9) If $6\mid \alpha$ and $p>3$ is odd, or $q=3^6$, then $G\cong
\psu$ or $G\cong \psu\cdot \langle \psi\rangle$,
 the extension of $\psu$ by a field automorphism of $\psu$, where $|\langle \psi\rangle|=2^i\cdot 3^j$ and $1\leq 2^i\cdot 3^j\leq (\alpha)_2\cdot (\alpha)_3$.\\
10) If $6\mid \alpha$ and $q=3^{\alpha}\neq 3^6$, then $G\cong \psu$
or $G\cong \psu\cdot \langle \varphi\rangle$,
 the extension of $\psu$ by a field automorphism of $\psu$, where $|\langle \varphi\rangle|=2^i$ and $2\leq 2^i\leq (\alpha)_2$.
\end{theorem}
\begin{proof}
By Theorem \ref{d=1}, $G$ has a normal series $1\unlhd H\unlhd
K\unlhd G$ such that $H$ and $G/K$ are $\pi_1$-groups, while $H$ is
nilpotent and $K/H\cong \psu$.

First we prove that $H=1$. Let $H\neq 1$ and $R$ be the Sylow
$r$-subgroup of $H$. Since $H\unlhd G$ and $H$ is nilpotent, so
$Z(R)\unlhd G$. Let $|Z(R)|=r^{\gamma}$. Since $Z(R)$ is an abelian
subgroup of $H$, so there exists a maximal abelian subgroup $N$ of
$G$ such that $Z(R)\leq N$. By assumption, $M(G)=M(\psu)$.
Consequently, $r^{\gamma}\mid |\psu|$. On the other hand,
$(q^2-q+1)/(3,q+1)$ is a divisor of $(|Z(R)|-1)$, by Lemma
\ref{admartabe-1}. If $q>5$, then by \cite[Lemma 2.6]{psl3od},
$r^{\gamma}-1\not\equiv 0\pmod{(q^2-q+1)/(3,q+1)}$, which is a
contradiction.

Let $q=4$. Therefore $(q^2-q+1)/(3,q+1)=13\mid (r^{\gamma}-1)$. Also
$r^{\gamma}\mid |{\rm {PSU}}_3(4)|=2^6\cdot 3\cdot 5^2\cdot 13$,
which is a contradiction. Consequently, $H=1$.

By Lemma \ref{=ordercom}, $G/K\leq {\rm Out}(\psu)$. So
$\pi(G/K)\subseteq \pi({\rm Out}(\psu))\subseteq \pi(df)$, where
$d=(3,q+1)$ is the order of the diagonal automorphism and $f=\alpha$
is the order of the field automorphism of $\psu$, by \cite{atlas}.

Let $d=3$ and $G\cong \psu\cdot d$, be the extension of $\psu$ by
the diagonal automorphism of $\psu$. We note that if $S$ is a simple
group of Lie type and $\tilde{S}=S\cdot d$, where $d$ is the
diagonal automorphism of $S$, then $\tilde{S}$ is a group of Lie
type in which the maximal tori $\tilde{T}$ have order $|T|d$, where
$T=\tilde{T}\cap S$ by \cite{lucido}. Therefore $3\sim r_{6}$ in
$\Gamma(\psu\cdot d)$, which is a contradiction, since $3\nsim
r_{6}$ in $\Gamma(G)$ by Remark \ref{2a2graph}.

Let $G\cong \psu\cdot \langle\theta\rangle$, the extension of $\psu$
by a field automorphism of $\psu$, where $\theta$ is of order $l$.

If $2\nmid \alpha$ and $3\nmid \alpha$, then by Theorem \ref{dfg},
$G\cong \psu$.

Consequently, $l=2^i\cdot3^{j}$, where $i,j\in\Bbb{N}\cup \{0\}$.

Note that \cite[Proposition 4.9.1]{gorlysol} asserts that all
automorphisms of prime order are conjugate. This implies that we can
just choose our favorite automorphism of order 2 and 3, and all
centralizers will have the same structure.

Using Theorem \ref{dfg}, we get the result.
\end{proof}
{\bf Acknowledgment}\\
The authors would like to thank Professor N. Gill and Professor G.
Malle for their help in the proof of Theorem \ref{malle}.

The second author would like to thank Institute for Research in Fundamental Sciences (IPM).

\end{document}